\documentclass[twoside,12pt,leqno]{amsproc}
\usepackage{amssymb,latexsym,enumerate,verbatim,stmaryrd,tikz}
\usepackage[colorlinks,linkcolor=purple,citecolor=blue,urlcolor=blue
,hypertexnames=true]{hyperref}
\usepackage{amsrefs}    
\usepackage{etoolbox}   
\hypersetup{citecolor=purple, linkcolor=blue, colorlinks=true}
\numberwithin{table}{section}

\theoremstyle{plain}
\newtheorem{theorem}{Theorem}[section]
\newtheorem{lemma}[theorem]{Lemma}

\newtheorem{corollary}[theorem]{Corollary}

\theoremstyle{definition} 
\newtheorem{definition}[theorem]{Definition}

\newtheorem{example}[theorem]{Example}
\AtEndEnvironment{remark}{\null\hfill$\triangle$}

\renewcommand{\geq}{\geqslant}  
\renewcommand{\leq}{\leqslant}  
\newcommand{\nonsplit}[2]{#1\raisebox{0.6ex}{$\cdot$} #2}

\newenvironment{claimproof}[1]{\par\noindent\underline{Proof:}\space#1}{\hfill $\blacksquare$}

\newcommand{\AGL}{\mathrm{AGL}}

\newcommand{\Aut}{\textup{Aut}}

\newcommand{\GL}{\mathrm{GL}}
\newcommand{\GaL}{\Gamma\mathrm{L}}

\newcommand{\Out}{\textup{Out}}

\makeatletter        
\def\@adminfootnotes{%
  \let\@makefnmark\relax  \let\@thefnmark\relax
  \ifx\@empty\@date\else \@footnotetext{\@setdate}\fi
  \ifx\@empty\@subjclass\else \@footnotetext{\@setsubjclass}\fi
  \ifx\@empty\@keywords\else \@footnotetext{\@setkeywords}\fi
  \ifx\@empty\thankses\else \@footnotetext{%
    \def\par{\let\par\@par}\@setthanks}%
  \fi}\makeatother   

\begin{document}

\hyphenation{}

\title[Bounds for Finite Semiprimitive Permutation Groups]{Bounds for Finite Semiprimitive Permutation Groups: Order, Base Size, and Minimal Degree}
\author{Luke Morgan, Cheryl E. Praeger, Kyle Rosa}

\address[Morgan, Praeger, Rosa]{
Centre for Mathematics of Symmetry and Computation,
The University of Western Australia,
35 Stirling Highway,
Crawley 6009, Australia.\newline
 Email: {\tt Luke.Morgan@uwa.edu.au} \newline
 Email: {\tt Cheryl.Praeger@uwa.edu.au}; \newline
 Email: {\tt Kyle.Rosa@research.uwa.edu.au}
 }

\date{\today}

\begin{abstract}
In this paper we study finite semiprimitive permutation groups, that is, groups in which each normal subgroup is transitive or semiregular. We give bounds on the order, base size, minimal degree, fixity, and chief length of an arbitrary finite semiprimitive group in terms of its degree. To establish these bounds, we classify finite semiprimitive groups that induce the alternating or symmetric group on the set of orbits of an intransitive normal subgroup.
\end{abstract}

\maketitle
\begin{center}{{\sc\tiny MSC 2010 Classification: 20B15, 20H30, 20B05}}\end{center}


\section{Introduction}\label{S1}
Semiprimitive groups have been the topic of several recent papers, driven by their connection to the Weiss Conjecture in graph theory \cites{OGRPG, CSPGGR, LSGTW}, the structure of collapsing monoids \cite{BM}, and the fact that they naturally generalise the important classes of \textit{primitive}, \textit{quasiprimitive}, and \textit{innately transitive} groups (see Definition \ref{primdef}), each of which have been studied extensively in their own right. The \textit{O'Nan-Scott-like} classifications of primitive, quasiprimitive, and innately transitive groups partition these classes into \textit{O'Nan-Scott types} according to the action of minimal normal subgroups. Recently, Giudici and Morgan \cite{GM} established a similar classification for semiprimitive groups, and posed a series of open problems, several of which are answered here. 

Throughout the paper, we adopt the following notation and terminology:
\begin{enumerate}
\item Let $\Omega$ be a set. We denote by $S_\Omega$ the \textit{symmetric group} of all permutations on $\Omega$, and refer to a subgroup of $S_\Omega$ as a \textit{permutation group} on $\Omega$. For a group $G$ acting (possibly unfaithfully) on a set $\Omega$, we denote by $\omega{}g$ the image of $\omega\in\Omega$ under $g\in{}G$.
\item A permutation group $G\leq{}S_\Omega$ is \textit{semiregular} if only the identity of $G$ fixes a point in $\Omega$. 
\item A permutation group $G\leq{}S_\Omega$ is \textit{transitive} if for all $\alpha,\omega\in\Omega$, there exists at least one $g\in{}G$ for which $\alpha{}g=\omega$. 
\end{enumerate}

\subsection{Semiprimitive Groups}  
 
As in \cite{GM}, we make the following definition.

\begin{definition}
A transitive permutation group $G\leq{}S_\Omega$ is \textit{semiprimitive} if every normal subgroup of $G$ is  transitive or semiregular. 
\end{definition}
A useful characterisation of semiprimitivity is given by the following lemma.

\begin{lemma}[\cite{BM}*{Lemma $2.4$}]\label{SPPS}
A finite transitive permutation group is semiprimitive if and only if for every intransitive normal subgroup $N$ of $G$,  the kernel of the action of $G$ on the set $\Delta$ of $N$-orbits is $N$.
\end{lemma} 

In this paper, we study finite semiprimitive groups, and bound several important group-theoretic quantities for this family. A key feature of these groups is their collection of \textit{antiplinths}.

\begin{definition}
Let $G$ be a permutation group. An \textit{antiplinth} $N$ of $G$ is a subgroup $N$ of $G$ that is maximal by inclusion among all intransitive normal subgroups of $G$.
\end{definition}

The term antiplinth was invented as a pun on the term \textit{plinth}, defined in \cite{GM}*{Definition $3.4$} to be a minimally transitive normal subgroup. An antiplinth, being maximal among intransitive normal subgroups, is in some sense the opposite. The importance of this concept to the study of semiprimitive groups, and its appearance in Theorem \ref{bounds}, is explained by the following lemma.

\begin{lemma}
\label{SMAQP}
Let $G\leq{}S_\Omega{}$ be a semiprimitive permutation group with antiplinth $N$, and let $\Delta$ be the set of $N$-orbits in $\Omega$. Then $G/N$ acts faithfully and quasiprimitively on $\Delta$.
\end{lemma}

\subsection{Bounds}
The \textit{order} $|G|$ of a group $G$ is the number of elements in it. The question of bounding this quantity for a class of groups satisfying a given property is very natural, and bounds on the orders of various kinds of permutation groups have been sought for more than a century. In fact this problem for primitive groups was the theme of the 1860 \emph{Grand Prix de Math\'ematiques} announced in 1857 by the Paris Academy (see, for example, the essay of Peter M.~Neumann in \cite[pp.3--4]{PMN}).

The base size $b(G)$, minimal degree $m(G)$, and fixed point ratio $fpr(G)$ are each important in computing applications, and also yield results about $|G|$. The \textit{chief length} of a group is an important theoretical invariant, and our work here furthers prior work on the \textit{composition length} of such groups in \cite{Me}. We give the formal definitions of each of these quantities in Definitions \ref{defc} and \ref{defb}. Our main result about semiprimitive groups is the following, answering Problems $1$, $3$, and $4a$ of \cite{GM}. 

\begin{theorem}\label{bounds}
Let $G\leq{S_\Omega}$ be a semiprimitive permutation group of degree $n$ such that $A_\Omega \nleqslant G$. The following hold:
\begin{enumerate}
	\item The order of $G$ satisfies $|G|<4^n$.
	\item There is a constant $n_1$ such that if $n\geq{}n_1$, then  $b(G)\leq{}4\sqrt{n}\log(n)$. For such groups we have $|G|\leq{}2^{4\sqrt{n}(\log_2n)^2}$.
    \item The minimun degree satisfies $m(G)\geq{}(\sqrt{n}-1)/2$.
    \item Either, for some antiplinth $M$ of $G$, the socle of $G/M$ is a product of alternating groups, or $\text{fpr}(G)\leq{}4/7$.
    \item The chief length of $G$ satisfies $l(G)\leq{}2\log_2n$.
\end{enumerate}
\end{theorem}

The result for minimal degree is particularly interesting, since by \cite{KPS}, groups with large minimal degree are \textit{indistinguishable}, in the sense that they prevent certain quantum algorithms from solving the \textit{hidden subgroup problem} efficiently. 

\begin{corollary}
Semiprimitive groups are \textit{indistinguishable}.
\end{corollary}

Note that the composition length of finite semiprimitive groups has been bounded previously by Glasby et al.~in \cite{Me}. For the present paper, we chose to generalise results that do not rely on the Classification of Finite Simple Groups (CFSG), and as such sharper bounds are possible for some of these quantities. In several cases, similar arguments to the ones used here also work to generalise the stronger results. 
For example, Liebeck and Saxl  \cite[Corollary 3]{bettermed} establish a bound for $m(G)$ for primitive groups $G$ that is a factor of four better than the function in Theorem \ref{bounds}$(3)$. The same bound can also be shown to hold for all semiprimitive groups, except $3.S_6$ acting on 18 points and the full alternating and symmetric groups.

\subsection{Classification} In order to establish the bounds in Theorem \ref{bounds}, we require knowledge of finite semiprimitive groups that induce the alternating or symmetric groups on the set of orbits of an intransitive normal subgroup. These are exactly the groups which cause problems for reductions. The following theorem provides a classification of such groups.

\begin{theorem}\label{classcor}
Let $G\leq{}S_\Omega$ be a semiprimitive group with a minimal normal subgroup $M\cong{}T^d$, where $T$ is a simple group. Let $\Delta$ be the set of $M$-orbits in $\Omega$, and let $r:=|\Delta|$ be the size of $\Delta$. If $G^\Delta\geq{}A_\Delta$ and $r\geq{}5$, then at least one of the following holds:
\begin{enumerate}
    \item $d\geq{}r-2$, 
    \item $T\gtrsim{}A_{r-1}$, or
    \item one of the following holds:\begin{enumerate}
    		\item $|\Omega|=2^4\times{}7$, and $G\cong{}C_2^4\rtimes{}A_7$, 
            \item $|\Omega|=2^4\times{}8$, and $G\cong{}\AGL_4(2)$ or $\nonsplit{C_2^4}{A_8}$, 
            \item $|\Omega|=3\times{}6$, and $G\cong{}\nonsplit{C_3}{A_6}$ or $(\nonsplit{C_3}{A_6})\rtimes{}C_2$, or 
            \item $|\Omega|=3\times{}5=4^2-1$, and $G\cong{}\GL_2(4)$ or ${\GaL_2(4)}$. 
    \end{enumerate}
\end{enumerate} 
\end{theorem}

\section{Preliminaries}\label{S2}

In this section we will collect the foundational material for the proofs of Theorems \ref{classcor} and \ref{bounds}. This includes all relevant definitions, theorems from the literature, and proofs for lemmas which we will later need.

\subsection{Group Theoretic Quantities}

We introduce the definitions of the group theoretic terms that we are primarily interested in, and explain the quantities that we will later bound.

\begin{definition}\label{defo}\label{defc}
Let $G$ be a finite group.
\begin{enumerate}
\item The \textit{socle} soc$(G)$ of $G$ is the subgroup of $G$ generated by the set of all minimal normal subgroups of $G$.
\item A \textit{chief series} $\{N_i\}$ for $G$ is a set of normal subgroups of $G$ with the property that, for all $i$, we have $N_i\leq{}N_{i+1}$, and $N_{i+1}/N_i$ is a minimal normal subgroup of $G/N_i$.
\item The \textit{chief length} $l(G)$ of $G$ is the length of a chief series of $G$.
\end{enumerate}
\end{definition}

By the Jordan-H\"{o}lder Theorem, the chief length of a group $G$ depends only on the isomorphism type of $G$, and is independent of the chosen series. If $G\leq{}S_\Omega$ is a permutation group, we define its stabiliser subgroups, bases, and base size. We  also define the degree and fixity of its elements, and in turn the fixed point ratio and minimal degree of $G$.

\begin{definition}\label{defb}\label{deffix}\label{deffpr}\label{defm} 
Let $G\leq{}S_\Omega$ be a finite permutation group, $\omega$ a point in $\Omega$, $\delta$ a subset of $\Omega$, and $g$ a nontrivial element of $G$.
\begin{enumerate}
\item The \textit{stabiliser} $G_{\omega}$ of $\omega$ is the subgroup $\{g\in{}G\mid{}\omega{}g=\omega\}$ of $G$ of elements which fix $\omega$.
\item The \textit{setwise stabiliser} $G_{\delta}$ of $\delta$ is the subgroup $\{g\in{}G\mid{}\delta{}g=\delta\}$ of $G$ of elements which map points in $\delta$ to points in $\delta$.
\item The \textit{pointwise stabiliser} $G_{(\delta)}$ of $\delta$ is the subgroup of $G$ of elements which fix all points in $\delta$, $$\{g\in{}G\mid{} \forall\alpha\in\delta : \alpha{}g=\alpha\} := \bigcap_{\alpha\in\delta}G_\alpha.$$
\item A \textit{base} of $G$ is a subset $B$ of $\Omega$ for which $G_{(B)}=1$.
\item The \textit{base size} $b(G)$ of $G$ is the smallest size of a base for $G$.
\item The \textit{fixity} fix$(g)$ of $g$ is the number of points in $\Omega$ that are fixed by $g$, $$\text{fix}(g) :={ } \mid\{\omega\in\Omega\mid{}\omega{}g=\omega\}\mid.$$
\item The \textit{degree} $m(g)$ of $g$ is the number of points moved by $g$, $$m(g):=n-\text{fix}(g).$$
\item The \textit{fixed point ratio} fpr$(G)$ of $G$ is the maximum proportion of points in $\Omega$ that are fixed by a nontrivial element of $G$, $$\text{fpr}(G):=\max_{g\in{}G\backslash{1}}{}\frac{\text{fix}(g)}{n}.$$
\item The \textit{minimal degree} $m(G)$ of $G$ is the smallest number of points that are moved by a nontrivial element of $G$, $$m(G):= \min_{g\in{}G\backslash{1}}{ }m(g) = n(1-\text{fpr}(G)).$$
\end{enumerate}
\end{definition}

\subsection{Primitive, Quasiprimitive, and Innately Transitive Groups}

We give the definition of primitive, quasiprimitive, and innately transitive groups, followed by a theorem collecting many of the results on these structures from the literature. We will later show how these bounds can be used to establish similar results for semiprimitive groups.

\begin{definition}\label{primdef}
Let $G\leq{}S_\Omega$ be a finite transitive permutation group. 

\begin{enumerate}
\item $G$ is \textit{primitive} if there does not exist a nontrivial $G$-invariant partition of $\Omega$.
\item $G$ is \textit{quasiprimitive} if every nontrivial normal subgroup of $G$ is transitive on $\Omega$.
\item $G$ is \textit{innately transitive} if $G$ has a minimal normal subgroup $N$ that is transitive on $\Omega$.
\end{enumerate}
\end{definition}

Note that since a transitive permutation group acts transitively on the set of orbits of its normal subgroups, every normal subgroup $N$ of a primitive group $G$ must be transitive. This is since otherwise $\Omega$ would have a nontrivial $G$-invariant partition, namely the set orbits of $N$. This shows that all primitive groups are quasiprimitive. Similarly, since every finite quasiprimitive group has some minimal normal subgroup, which by the definition of quasiprimitivity must be transitive, all quasiprimitive groups are innately transitive. In an innately transitive group $G$, every intransitive normal subgroup of $G$ centralises a transitive minimal normal subgroup of $G$. Centralisers of transitive groups are known to be semiregular \cite[Theorem 3.2]{PraegerSchneider}. In particular, every normal subgroup of $G$ must be either transitive or semiregular, and hence innately transitive groups are seen to be semiprimitive. We are now ready to present the collection of results which will be generalised in later sections.

\begin{theorem}\label{QPD}\label{BQG}\label{qprimmed}\label{PrimB}\label{gmf}
Let $G\leq{S_\Omega}$ be an innately transitive group of degree $n$. The following hold:
\begin{enumerate}
	\item Either $|G|<4^n$, or $G\geq{}A_\Omega$. 
	\item There is constant $n_0$ such that if $n\geq{}n_0$, then either $b(G)\leq{}4\sqrt{n}\log(n)$, or $G\geq{}A_\Omega$. 
    \item Either $m(G)\geq{}(\sqrt{n}-1)/2$, or $G\geq{}A_\Omega$. 
    \item If $G$ is primitive then either the socle of $G$ is a product of alternating groups, or $\text{fpr}(G)\leq{}4/7$. 
    \item If $G$ is primitive then $l(G)\leq{}2\log_2n$. 
\end{enumerate}
\end{theorem}

\begin{proof}

We draw these results from the sources cited below:

\begin{enumerate}
\item See \cite{PrimO}*{Theorem 1} for $G$ primitive, \cite{BQG}*{Theorem $4.2$} for $G$ quasiprimitive, and \cite{JB}*{Theorem $6.1.4$} for $G$ innately transitive.
\item See \cite{PrimB}*{Theorem 0.2} for $G$ primitive, \cite{BQG}*{Theorem $5.2$} for $G$ quasiprimitive, and \cite{JB}*{Theorem $6.1.6$} for $G$ innately transitive.
\item See \cite{Wie}*{Theorem 15.1} and \cite{PrimB}*{Theorem $0.3$} for $G$ primitive, \cite{BQG}*{Theorem $7.2$} for $G$ quasiprimitive, and \cite{JB}*{Theorem $6.1.7$} for $G$ innately transitive.
\item See \cite{gmf}.
\item See \cite{NM}*{Theorem 10.0.0.6}.
\end{enumerate}
\end{proof}

\subsection{Structure Theory}

We require some results about the structure of certain group actions and group extensions.

\begin{lemma}\label{subcartesian}
Suppose $M\times{}H$ is a group which acts transitively on a set $\Omega$. Let $\omega\in\Omega$, and define $\delta:=\omega^M$, and $\sigma:=\omega^H$. The following hold:
\begin{enumerate}
    \item $(MH)_{\delta\cap{}\sigma}=M_\sigma{}H_\delta$.
    \item $M_\sigma$ and $H_\delta$  each act transitively on $\delta\cap{}\sigma$.
    \item $M_\sigma\cong{}(MH)_\omega/H_\omega$, and $H_\delta\cong{}(MH)_\omega/M_\omega$.
\end{enumerate}
\end{lemma}

\begin{proof}Note that since both $\delta$ and $\sigma$ are blocks for $MH$, so is $\delta\cap\sigma$. This implies that $\delta\cap{}\sigma$ is a block for each of $M$ and $H$ too. If $mh\in{}(MH)_{\delta\cap{}\sigma}$, then $\omega^{mh}\in{}(\delta\cap{}\sigma)^{mh}=\delta\cap\sigma$. Now, $h$ maps $\omega^m\in{}\sigma$ to $\omega^{mh}\in\delta\cap{}\sigma$, and in particular $\omega^m$ and $\omega^{mh}$ are in the same $H$-orbit, namely $\sigma$. Hence $m\in{}M_\sigma$. A similar argument shows that $h\in{}H_\delta$, and hence that $mh\in{}M_\sigma{}H_\delta$. So we have $(MH)_{\delta\cap{}\sigma}\leq{}M_\sigma{}H_\delta$. Conversely, since $M=M_\delta$ we have $M_{\delta\cap{}\sigma}=M_\delta\cap{}M_\sigma=M\cap{}M_\sigma=M_\sigma$, and similarly $H_\delta=H_{\delta\cap{}\sigma}$. Hence $M_\sigma{}H_\delta\leq{}(MH)_{\delta\cap\sigma}$. This shows that $M_\sigma{}H_\delta=(MH)_{\delta\cap{}\sigma}$, establishing part $(1)$ of the Lemma.

Let $\alpha$ and $\beta$ be elements of $\delta\cap{}\sigma$. Since $M$ is transitive on $\delta$, there exists some permutation $m\in{}M$ such that $\alpha^m=\beta$. As $\delta\cap{}\sigma$ is a block of imprimitivity for $G$, and $m$ maps at least one element of $\delta\cap{}\sigma$ back into $\delta\cap{}\sigma$, we conclude that $m$ fixes $\delta\cap{}\sigma$ setwise, and find that $m\in{}M_{\delta\cap{}\sigma}$. Hence $M_{\delta\cap{}\sigma}$ acts transitively on $\delta\cap{}\sigma$. As $M_{\delta\cap{}\sigma}=M_\delta\cap{}M_\sigma=M\cap{}M_\sigma=M_\sigma$, we conclude that $M_\sigma$ acts transitively on $\delta\cap{}\sigma$. An identical argument shows that $H_\delta$ also acts transitively on $\delta\cap\sigma$. This establishes part $(2)$ of the lemma.

Since $H\cap{}M=1$, we have $M_\sigma\cong{}HM_\sigma/H=(MH)_\sigma/H$. As $H$ is transitive on $\sigma$ and $\omega\in\sigma$, we have that $(MH)_\sigma=H(MH)_\omega$. Hence by the third isomorphism theorem, we have $$\frac{(MH)_\sigma}{H}=\frac{H(MH)_\omega}{H}\cong{}\frac{(MH)_\omega}{H\cap{}(MH)_\omega}=\frac{(MH)_\omega}{H_\omega}.$$ Hence $M_\sigma\cong{}(MH)_\omega/H_\omega$, as required. By symmetry, $H_\delta\cong(MH)_\omega/M_\omega$. This establishes part $(3)$ of the lemma.
\end{proof}

We will need to address the central extensions of the alternating and symmetric groups in some detail, and so we present the relevant information here.

\begin{definition}\label{schur}
If $G$ is a group, we let $\mathbb{M}(G)$ denote the \textit{Schur multiplier} $H^2(G,\mathbb{Z})$ of $G$.
\end{definition}

\begin{lemma}[]
If $H$ and $G$ are groups, and $\pi:H\to{}G$ is a surjection with the property that $\ker\pi\leq{}Z(H)\cap{}H'$, then $\ker\pi\lesssim\mathbb{M}(G)$. Such a group $H$ is called a $\textit{stem extension}$ of $G$.
\end{lemma}

This definition and lemma come from \cite{Rotman}*{p.205} and \cite{Rotman}*{p.208}, respectively. Finally, we bring together some facts about the action of an abstract group on a minimal normal subgroup, and present a collection of results summarising the important features.

\begin{lemma}\label{issol}
Let $G$ be a group and $M\cong{}T^d$ be a nonabelian minimal normal subgroup of $G$, where $T$ is a (possibly abelian) simple group and $d$ is a positive integer. Let $\Gamma$ be the set of simple direct factors of $M$, $\phi:G\to{}S_\Gamma$ be the induced action of $G$ on $\Gamma$, and $N$ be the kernel of $\phi$. Then $N/MC_G(M)$ is soluble.
\end{lemma} \begin{proof} Note that $d=|\Gamma|$. Consider the conjugation action $\psi:G\to{}\Aut(M)$ of $G$ on $M$. Now, $\Aut(M)\cong{}\Aut(T)^{d}\rtimes{S_\Gamma}$, and so we have the following inequalities:\begin{enumerate}
    \item{}$\ker\psi=C_G(M)<MC_G(M)\leq{}N$,
    \item{}$\psi(MC_G(M))=\psi(M)=\text{Inn}(T)^d$,
    \item{}$\psi(N)\leq{}\text{Aut}(T)^d.$ 
\end{enumerate} 
Now, as $C_G(M)\leq{}MC_G(M)\leq{}N$, we have by the third and first isomorphism theorems that $$\frac{N}{MC_G(M)}\cong{}\frac{N/C_G(M)}{MC_G(M)/C_G(M)}=\frac{\psi(N)}{\psi(MC_G(M))}\leq{}\frac{\Aut(T)^d}{\text{Inn}(T)^d}\cong{}\text{Out}(T)^d.$$ By the Schreier Conjecture, $\Out(T)$ is soluble, and hence so is $\Out(T)^d$. Subgroups of soluble groups are soluble, and so we conclude that $N/MC_G(M)$ is soluble, as required. \end{proof}

\begin{corollary}\label{mnsaction}
Let $M\trianglelefteq{G}$ be a minimal normal subgroup of $G$, and $H:=C_G(M)$. Then one of the following holds:
\begin{enumerate}
    \item $M$ is abelian, $M\leq{}H$, there exists a prime $p$ and positive integer $d$ such that $M\cong{}C_p^d$, and $G/H$ acts on $M$ irreducibly and faithfully,
    \item $M$ is nonabelian, there exists a nonabelian simple group $T$ and positive integer $d$ such that $M\cong{}T^d$, and $G/MH$ acts transitively on the $d$ simple direct factors of $M$ with a soluble kernel.
\end{enumerate}
\end{corollary}
\begin{proof}
If $M$ is abelian, then since $M$ is minimal normal in $G$, we have $M\cong{}C_p^d$ for some prime $p$ and positive integer $d$. The condition that $M$ is minimal normal in $G$ is equivalent to the requirement that $G$ acts irreducibly on $M$ by conjugation. Let $\rho:G\to{}\text{Aut}(M)$ be the conjugation action of $G$ on $M$. The kernel of $\rho$ is $C_G(M)=H$, and so by the first isomorphism theorem, $\rho(G)\cong{}G/H$.

If $M$ is nonabelian, then $M\cap{}H=1$, and so $MH\cong{}M\times{}H$. Since $M$ is minimal normal in $G$, $M\cong{}T^d$ for some nonabelian simple group $T$ and positive integer $d$, and $G$ acts transitively by conjugation on the set $\Gamma:=\{T_1, ..., T_d\}$ of simple direct factors of $M$. Let $\phi:G\to{}S_\Gamma$ be this action. Let $N$ be the kernel of $\phi$. By Lemma \ref{issol}, $MH\trianglelefteq{N}$, and $N/MH$ is soluble. As $MH\trianglelefteq{N}$, $\phi$ induces an action $\overline{\phi}$ of $G/MH$ on $\Gamma$, given by $\overline{\phi}(MHg):=\phi(g)\in{}S_\Gamma$. Note that this map is well defined since $MH\leq{}N$. 

Now, $MHg\in{}\ker\overline\phi$ if and only if $1=\overline\phi(MHg)=\phi(g)$, or equivalently if $g\in{}\ker\phi=N$. Hence $\ker\overline\phi=N/MH$. As $N/MH\cong{}\overline\phi(N/MH)$ is soluble, this establishes part $(2.)$ of the theorem.
\end{proof}

\subsection{Properties of $A_\Omega$ and $S_\Omega$}

We present a collection of results on the actions and extensions of the alternating and symmetric groups that will be used in the classification of their semiprimitive covers.

\begin{lemma}\label{ansnrep}
Let $G=A_n$ or $S_n$. If $G$ acts transitively and faithfully on $d$ points, then $d\geq{}n$. If $G$ acts irreducibly and faithfully on $C_p^d$, with $p$ a prime, then either $(C_p^d, G)=(C_2^4,A_7)$ or $(C_2^4, A_8)$, or $d\geq{}n-2$.
\end{lemma}

\begin{proof}
Note that if $\rho:A_n\to{}S_d$ is a faithful action of $A_n$ on $d$ points, then $|A_n|\leq|S_d|$, which implies that $n!\leq{}2d!$. This implies that $n\leq{}d$, as required. A similar argument shows that the same bound holds for faithful actions of $S_n$.
If $G$ acts faithfully and irreducibly on $C_p^d$, then by the results of \cites{mrd2, mrd1}, either $(C_p^d, G)=(C_2^4,A_7)$, $(C_2^4, A_8)$, or $d\geq{}n-2$.
\end{proof}

We now wish to establish which central covers of $A_n$ or $S_n$ have faithful actions that induce $A_n$ or $S_n$ on the orbits of a normal subgroup. Recall from Definition \ref{schur} that $\mathbb{M}(G)$ is the Schur multiplier of $G$, and classifies the isomorphism types of stem extensions of $G$. See \cite{TFSG}*{Chapter 2.7.2} for a proof of the following two lemmas.

\begin{lemma}\label{ansnmult}
Let $G=A_n$ or $S_n$. Then $\mathbb{M}(G)=C_2$, except if $n\leq3$, in which case $\mathbb{M}(G)=1$, or if $G=A_6$ or $A_7$, in which case $\mathbb{M}(G)=C_6$.\end{lemma}

Since we are only interested in minimal extensions, this lemma implies we need only consider covers with a kernel of size $2$ or $3$. In order to further examine the double-covers, we use the following lemma.

\begin{lemma}[]\label{2ansn} Let $G=A_n$ or $S_n$ and $n\geq{}5$. There exists a central extension $\nonsplit{C_2}{G}$ that may be thought of as a set of $\pm$-signed permutations. This group is uniquely determined for $A_n$, but there exist two non-isomorphic such groups for $S_n$. We denote these groups by $\nonsplit{C_2}{S_n}^+$ and $\nonsplit{C_2}{S_n}^-$. There exist natural projections $\pi:\nonsplit{C_2}{G}\to{}G$ with kernel $Z:=Z(\nonsplit{C_2}{G})$ that simply ``forget the sign" of a given element. For a given $n$, each of $\nonsplit{C_2}{A_n}$, $\nonsplit{C_2}{S_n^+}$, and $\nonsplit{C_2}{S_n^-}$ have isomorphic derived subgroups, namely $(\nonsplit{C_2}{G})'=\nonsplit{C_2}{A_n}$. Finally, any signed double transposition $z\in{}\nonsplit{C_2}{A_n}=(\nonsplit{C_2}{G)'}$ satisfies $z^2=-1\in{}Z$.\end{lemma}

Using the above result, we can prove the following.

\begin{lemma}\label{no2n}
Let $G\cong{}\nonsplit{C_2}{A_n}$ or $\nonsplit{C_2}{S_n^\pm}$, for some $n\geq{}5$. There is no faithful action of $G$ on $2n$ points.\end{lemma} \begin{proof}
We adopt the notation of Lemma \ref{2ansn}. Namely, we let $Z$ be the center of $G$, $z$ be a generator of $Z$, and $\pi:G\to{}G/Z\leq{}{}S_n$ is the natural projection of $G$ onto $S_n$. We proceed by assuming $\rho$ is a faithful permutation representation of $G$ on $2n$ points, and establishing a contradiction. Note that since the natural projection $\pi$ is a mapping into $S_n$, we may view $\pi$ as an unfaithful permutation action of $G$ on $n$ points. As $Z$ is a normal subgroup, any faithful action $\rho:G\to{}S_{2n}$, induces an action $\overline{\rho}:G\to{}S_n$ on the orbits of $Z\cong{}C_2$, with kernel $Z$. Note that since $Z$ is cyclic of prime order, it must be semiregular, and hence

$$\overline{\rho}(G_\omega)\cong{}\frac{ZG_\omega{}}{Z}\cong{}\frac{G_\omega}{G_\omega\cap{Z}}=\frac{G_\omega}{1}.$$

If $n\ne{}6$, then there exist unique faithful actions of $A_n$ and $S_n$ on $n$ points, and so $\overline\rho:G\to{}G/Z\leq{}S_n$ must be $S_n$-conjugate to the natural projection $\pi$. That is, there exists some $\sigma\in{}S_n$ for which $\overline{\rho}(g)=\pi(g)^\sigma$ for all $g\in{}G$. This establishes a permutational isomorphism between $\overline{\rho}$ and $\pi$. Let $\omega$ be a point in the $2n$-element set $G$ acts on. As $\overline{\rho}$ is permutationally isomorphic to $\pi$, and $r\geq{}5$, we may assume that $\overline{\rho}(G_\omega)\cong{}A_{r-1}\geq{}A_4$, and hence contains a double transposition. This implies that $G_\omega$ contains a signed double transposition $z$. By Lemma \ref{2ansn}, the square of any signed double transposition is the central involution $-1$. As $z\in{}G_\omega$ and $z^2=-1$, we have that $Z\leq{}G_\omega$. But then since $Z$ is normal, it must stabilise every point, contradicting that $\rho$ was faithful. Hence, $n=6$.

By inspection, the minimal degree of a faithful representation of $\nonsplit{C_2}{A_6}$ or $\nonsplit{C_2}{S_6^\pm}$ is  $80$, and hence these groups cannot induce an action of degree $6$ on blocks of size $2$. Thus no faithful action of $G$ on $2n$ points exists.
\end{proof}

\subsection{Numerical Lemmas}

Finally, we need two simple numerical results, which we establish now.

\begin{lemma}\label{indlem}
If $m>r\geq{}5$, then $mr!<4^{mr}$.
\end{lemma}
\begin{proof}
We proceed by induction on $r$. If $r=5$ then $mr!<4^{mr}$.
Here, $mr!=120m$ and $4^{mr}=(4^5)^m=1024^m$. Clearly $1024^m>120m$ for all $m>5$, and hence the base case is established. For the inductive step, we assume $mr!<4^{mr}$ and show $m(r+1)!<4^{m(r+1)}$. Now, by the inductive hypothesis, 
\begin{align*} m(r+1!)&=mr!(r+1)<4^{mr}(r+1)=4^{m(r+1)}\left(\frac{r+1}{4^m}\right),\end{align*}
and since $m>r$, we have $$4^{m(r+1)}\frac{r+1}{4^m}\leq{}4^{m(r+1)}\left(\frac{r+1}{4^r}\right)<4^{m(r+1)},$$
and so by induction the lemma holds for all $m>r\geq{}5$.
\end{proof}

\begin{lemma}\label{numlem2}
If $a,b\geq{}2$, then either ${a(\sqrt{b}-1)\geq{}\sqrt{ab}-1}$, or $b=2$ and $a\leq{}5$. 
\end{lemma}
\begin{proof} We proceed by solving the inequality for $b$, and then bounding the values of $a$ and $b$ which do not satisfy the inequality.

\noindent\underline{Claim 1:}
$a\sqrt{b}-a\geq{}\sqrt{ab}-1$ if and only if $b\geq{}\left(1+\frac{1}{\sqrt{a}}\right)^2.$
\begin{claimproof}
Assume that $a\sqrt{b}-a\geq{}\sqrt{ab}-1$. Rearranging, we have $\sqrt{b}(a-\sqrt{a})\geq{}a-1$, and so
$$\sqrt{b}\geq{}\frac{a-1}{a-\sqrt{a}}=\frac{(\sqrt{a}+1)(\sqrt{a}-1)}{\sqrt{a}(\sqrt{a}-1)}=\frac{\sqrt{a}+1}{\sqrt{a}}=1+\frac{1}{\sqrt{a}}.$$ Hence $b\geq{}(1+\frac{1}{\sqrt{a}})^2$, as required. Each implication in this demonstration is bidirectional, and so the converse also holds, establishing the claim.
\end{claimproof}

\noindent\underline{Claim 2:}
If $a\geq{}6$ or $b\geq{}3$, then $a\sqrt{b}-a\geq{}\sqrt{ab}-1$.
\begin{claimproof}
If $a\geq{}6$, then since $\left(1+\frac{1}{\sqrt{a}}\right)^2$ is decreasing we have $$\left(1+\frac{1}{\sqrt{a}}\right)^2\leq{}\left(1+\frac{1}{\sqrt{6}}\right)^2<2\leq{}b,$$ and so by Claim $1$ we have that $a\sqrt{b}-a\geq{}\sqrt{ab}-1$, as required. If $b\geq{}3$, then since $a\geq{}2$ and $\left(1+\frac{1}{\sqrt{a}}\right)^2$ is decreasing, we have $$\left(1+\frac{1}{\sqrt{a}}\right)^2\leq{}\left(1+\frac{1}{\sqrt{2}}\right)^2<3\leq{}b,$$ and so by Claim $1$ we have that $a\sqrt{b}-a\geq{}\sqrt{ab}-1$, as required.\end{claimproof} \\

Note that Claim $2$ is logically equivalent to the statement of the lemma, and so the result is proven.\end{proof}

\section{Covers}\label{S3}

The goal of this section is to establish a classification of semiprimitive groups which induce the alternating or symmetric group on the set of orbits of a minimal normal subgroup. In order to do this, we first prove a lemma about the structure of semiprimitive groups in terms of the centralisers of their antiplinths.

\subsection{Covers of Semiprimitive Groups}\label{spn}

For the remainder of this section, we adopt the following notation, which applies to an arbitrary semiprimitive group with a minimal normal antiplinth:

\begin{enumerate}
    \item $G\leq{}S_\Omega$ is a finite semiprimitive group with minimal normal antiplinth $M$,
    \item $M\cong{}T^d$ for some simple group $T$ and positive integer $d$,
    \item $H:=C_G(M)$ is the centraliser of $M$,
    \item $\Delta:=\Omega/M$ and $\Sigma:=\Omega/H'$ are the sets of orbits of $M$ and $H'$, respectively, 
    \item $\omega\in\Omega$, and $\delta:=\omega^M\in\Delta$ and $\sigma:=\omega^{H'}\in\Sigma$ are the $M$- and $H'$-orbits of $\omega$,
    \item $m:=|M|$ is the order of $M$, and
    \item $r:=|\Delta|$ is the number of $M$-orbits.
\end{enumerate}

Note that as $G$ is semiprimitive and $M$ is an intransitive normal subgroup of $G$, it follows from Lemma \ref{SPPS} that $G^\Delta\cong{}G/M$.

\begin{lemma}\label{trilema}
Suppose $MH>{}M$. Then, exactly one of the following holds:
\begin{enumerate}
    \item $M\leq{}Z(H)\cap{}H'$, $H'$ is transitive, and $M\lesssim{}{\mathbb{M}}(H^\Delta)$,
    \item $M\leq{}Z(H)\cap{}H'$, $H'$ is intransitive, $M\lesssim{}{\mathbb{M}}(H^\Delta)$, and $H^\Delta$ is an elementary abelian $q$-group for some prime $q$,
    \item $M\cap{}H'=1$, $H'$ is transitive, and $M\cong{}(H')_\delta/(H')_\omega$, or
    \item $M\cap{}H'=1$, $H'$ is intransitive, and $M_\sigma\cong{}(H')_\delta$.
\end{enumerate}
\end{lemma}
\begin{proof}
As $H'$ is a characteristic subgroup of $H$, which in turn is a normal subgroup of $G$, $H'$ is normal in $G$. Since $M$ is a minimal normal subgroup of $G$, and $H'$ is a normal subgroup of $G$, we have that $M\cap{}H'=M$ or $1$.

If $M\cap{}H'=M$, then as $H=C_G(M)$, we have $M\leq{}Z(H)$, and hence $M\leq{}Z(H)\cap{}H'$. Hence $G$ is a stem extension of $M$, and so $M\lesssim{}\mathbb{M}(H/M)=\mathbb{M}(H^\Delta)$. Note in particular that $M$ is abelian. If $H'$ is transitive we are done, and this is case $(1.)$ of the Lemma. If $H'$ is intransitive, then since $M$ is an antiplinth and $M\leq{}H'$, we have $M=H'$. Now $H^\Delta\cong{}H/M=H/H'$, and hence $H^\Delta$ is abelian. By Lemma \ref{SMAQP}, $G^\Delta$ is quasiprimitive. As $H^\Delta$ is a normal subgroup of $G^\Delta$, it must be that $H^\Delta$ is either transitive or trivial. If $H^\Delta$ is trivial, then $H\leq{}M$, and hence $H$ is abelian. This implies that $M=H'=1$, contradicting that $M$ is nontrivial. Hence we may assume that $H^\Delta$ is transitive, and since $H^{\Delta}$ is abelian, it is hence regular. By the O'Nan-Scott Theorem \cite{OnanScott}, this implies that $G^\Delta$ is an affine quasiprimitive group with socle $H^\Delta$, and hence $G^\Delta\cong{}{C_q^b}\rtimes{}(G^\Delta)_\delta$ for some prime $q$, integer $b$, and block $\delta\in\Delta$. In particular, $H^\Delta\cong{}C_q^b$ is an elementary abelian $q$-group. This establishes case $(2.)$ of the Lemma.

If $M\cap{}H'=1$, then as $M$ and $H'$ centralise each other, we have $MH'\cong{}M\times{}H'$. If $H'$ is transitive, then $\Omega=\omega^{H'}=\sigma$. This implies that $M_\sigma=M$, and so by part $(3.)$ of Lemma \ref{subcartesian} applied to $M\times{}H'$, we have $M=M_\sigma\cong{}(MH')_\omega/(H')_\omega={(H')_\delta}/{(H')_\omega}$. This is case $(3.)$ of the Lemma. If $H'$ is intransitive, then as a normal subgroup of a semiprimitive group, it is semiregular. If $H'$ is nontrivial, then since $M$ is an antiplinth, $MH'$ must be transitive. Now, $M_\omega=1=(H')_\omega$ and so by part $(3.)$ of Lemma \ref{subcartesian} we have $M_\sigma\cong{}(MH')_\omega\cong{}(H')_\delta.$ If $H'$ is trivial, then $H'=1$, so $(H')_\delta=1$ and $\sigma=\omega^{H'}=\{\omega\}$. Hence $M_\sigma=M_\omega=1=(H')_\delta$. This establishes case $(4.)$ of the Lemma.
\end{proof}

If $(H')_\delta$ has a unique minimal normal subgroup, as will turn out to be the case when we restrict our attention to the covers of alternating and symmetric groups in the next section, then we can say more about the isomorphism class of $T$. 

\begin{definition}
A group is \textit{monolithic} if it has a unique minimal normal subgroup.
\end{definition}

\begin{lemma}\label{tricor}
In case $(4.)$ of Lemma \ref{trilema}, if $(H')_\delta$ is monolithic then $T$ contains a subgroup isomorphic to $(H')_\delta$.
\end{lemma}
\begin{proof}
Consider the projection mappings $\pi_i:M\to{}T_i$ from $M$ to its simple direct factors $\{T_1, ..., T_d\}$. Note that $M_\sigma=\{(\pi_1(g), ..., \pi_d(g))\mid{}g\in{}M_\sigma\}$. In particular, if, for some $g\in{}M$, $\pi_i(g)=1$ for all $i=1, ..., d$, then it must be that $g=1$. To state the contrapositive, if $g\ne{}1$, then there exists some $i$ for which $\pi_i(g)\ne{}1$. Since $M_\sigma\cong{}(H')_\delta$ in case $(4.)$ of Lemma \ref{trilema}, $M_\sigma$ has a unique minimal normal subgroup. Choose a nontrivial element $g$ in the unique minimal normal subgroup of $M_\sigma$. As $\pi_i(g)\ne1$ for some $i$, $\pi_i|_{M_\sigma}$ must be faithful, since the unique minimal normal subgroup of $M_\sigma$ is not in the kernel. Hence $M_\sigma\cong\pi_i(M_\sigma)\leq{}T_i\cong{}T$. As $M_\sigma\cong{}(H')_\delta$, we have $(H')_\delta\lesssim{}T$.
\end{proof}

\subsection{Extensions of $A_\Omega$ and $S_\Omega$}

We now use Lemma \ref{trilema} to classify the groups which induce $A_\Delta$ or $S_\Delta$ on the set $\Delta$ of orbits of a minimal normal subgroup. We restrict our attention to the case where $|\Delta|\geq{}5$. This is due to the fact that under this assumption $A_{\Delta}$ is simple and $S_\Delta$ is almost simple, which simplifies the analysis. Even with this restriction, many interesting low-degree phenomena occur. The non-simplicity of ${A_5}$'s point stabiliser, the exceptional triple covers of $A_6$ and $S_6$, and the exceptional isomorphism $A_8\cong{}\GL_2(4)$ each give rise to interesting examples of semiprimitive groups. There are also several infinite families of covers, as detailed in the following theorem. 

\begin{theorem}
Suppose $G^\Delta\geq{}A_\Delta$ and $r\geq{}5$. Then, one of the following holds:

\begin{enumerate}\label{ansnclassification}
    \item $H^\Delta=1$, and one of the following hold:
    \begin{enumerate}
        \item $M$ is abelian, and either $d\geq{}r-2$, or $(M,G^\Delta)=(C_2^4,A_7)$ or $(C_2^4,A_8)$, 
        \item $M$ is nonabelian, and $d\geq{}r$.
    \end{enumerate}
    \item $H^\Delta=G^\Delta$, and one of the following holds:
    \begin{enumerate}
        \item $G\cong{}\nonsplit{C_3}{A_6}$,
        \item $G\cong{}A_{r-1}\times{}G^\Delta$ or $C_3\times{}A_5$, and $G$ is innately transitive,
        \item $G\cong{}T\times{}G^\Delta$, and $T$ contains a subgroup isomorphic to $A_{r-1}$.
    \end{enumerate}
    \item $H^\Delta=A_\Delta$ and $G^\Delta=S_\Delta$, and one of the following holds: 
    \begin{enumerate}
        \item $G\cong{}(\nonsplit{C_3}{A_6})\rtimes{}C_2$,
        \item $G\cong{}(A_{r-1}\times{}A_r)\rtimes{}C_2$ or $(C_3\times{}A_5)\rtimes{}C_2$, and $G$ is innately transitive,
        \item $G\cong{}(T\times{}A_r)\rtimes{}C_2$ or $(T^2\times{}A_r)\rtimes{}C_2$, and $T$ contains a subgroup isomorphic to $A_{r-1}$.
    \end{enumerate}
\end{enumerate}
\end{theorem}

\begin{proof}
Note that $r\geq{}5$, and hence that $A_\Delta$ is simple. As $H^\Delta$ is a normal subgroup of $G^\Delta$, the possibilities for $(H^\Delta, G^\Delta)$ are $(1, A_r)$, $(1, S_r)$, $(A_r, S_r)$, $(A_r, A_r)$, and $(S_r, S_r)$. These possibilities can be partitioned according to whether $G^\Delta/H^\Delta$ is equal to $G^\Delta$, $C_2$, or $1$. We analyse each case and study the implications of Lemma \ref{mnsaction} and Lemma \ref{trilema}.

To begin, note that since $M$ is minimal normal and $H'$ is normal, either $M\leq{}H'$ or $M\cap{}H'=1$. If $M\leq{}H'$, then since $M$ is an antiplinth, either $M=H'$, or $M<H'$ and $H'$ is transitive. Also, as $H=C_G(M)$, we must have that $M\leq{}Z(H)$, and hence $M\leq{}Z(H)\cap{}H'$. If $M\cap{}H'=1$, then $H'$ may be either transitive (which is the case if $G$ is innately transitive, for example), or intransitive.

\noindent\underline{Claim 1:}
If $M\cap{}H'=1$, then $(H')_\delta\cong{}((H^\Delta)')_\delta$.

\begin{claimproof}
Since $M\cap{}H'=1$, we have that $(H')_\delta\cong{}(H')_\delta{}M/M\cong{}((H')_\delta)^\Delta$. Clearly $((H')_\delta)^\Delta=((H')^\Delta)_\delta$, and since $h\mapsto{}h^\Delta$ is a homomorphism, we also have $(H')^\Delta=(H^\Delta)'$. Putting everything together, we find that $(H')_\delta\cong{}(H^\Delta)'_\delta$, as required.
\end{claimproof}

\noindent\underline{Claim 2:} If $H^\Delta=1$, then Part $(1.)$ of the theorem holds.
\begin{claimproof}
If $H^\Delta=1$, then $H\leq{}M$, and so $MH=M$ and $G/MH=G/M\cong{}G^\Delta$. By Corollary \ref{mnsaction}, $G/H$ acts faithfully on $M$. If $M$ is abelian, then since $M$ is minimal normal in $G$, we have $M\cong{}C_p^d$ for some prime $p$, and $G$ acts irreducibly on $M$ with kernel $H$. Since $M$ is abelian, $M\leq{}H$, and as $H\leq{}M$ we must have that $H=M$. Hence $G^\Delta\cong{}G/M=G/H$ acts irreducibly and faithfully on $M$. Now, by Lemma \ref{ansnrep}, either $(M,G^\Delta)=(C_2^4,A_7)$ or $(C_2^4,A_8)$, or $d\geq{}r-2$, as required. Suppose now that $M$ is nonabelian. By Corollary \ref{mnsaction}, $G^\Delta$ acts transitively and with a soluble kernel on the set $\Gamma$ of $d$ simple direct factors of $M$. As $r=|\Delta|\geq{}5$, $G^\Delta$ has no nontrivial soluble normal subgroups, and hence $G^\Delta$ acts faithfully on $\Gamma$, and by Lemma \ref{ansnrep}, $d\geq{}r$. 
\end{claimproof}

\noindent\underline{Claim 3:} If $H^\Delta=G^\Delta$, then Part $(2.)$ of the theorem holds.

\begin{claimproof}
If $H^\Delta=G^\Delta$, then by the third and first isomorphism theorems, $$\frac{G}{MH}\cong{}\frac{G/M}{MH/M}=\frac{G^\Delta}{H^\Delta}\cong1,$$ and so $G=MH$. If $K$ is a normal subgroup of $M$, then it is normalised by both $M$ and $H=C_G(M)$, and hence is normal in $G=MH$. As $M$ is minimal normal, this implies that $K=M$. Hence $M$ is simple, which is to say that $d=1$. Since $M$ is intransitive, this implies that $G=MH>M$, and so Lemma \ref{trilema} applies. We consider each of the cases of Lemma \ref{trilema} separately. Note that case $(2.)$ of Lemma \ref{trilema} cannot hold, as $G^\Delta$ has no nontrivial elementary abelian normal subgroups.

Suppose case $(1.)$ of Lemma \ref{trilema} holds. Then $M\lesssim{\mathbb{M}(H/M)}=\mathbb{M}(H^\Delta)$. Since $r\geq{}5$, we have by Lemma \ref{ansnmult} that either $\mathbb{M}(H^\Delta)=2$, or $r=6$ or $7$ and $\mathbb{M}(H^\Delta)=6$. As $n\geq{}5$, Lemma \ref{no2n} implies that the double covers of $A_r$ and $S_r$ do not have faithful actions on $2r$ points. That is, if $M\cong{}C_2$ then $G$ could not induce $A_r$ or $S_r$ on the set of $M$-orbits. This leaves only the possibilities $r=6$ or $7$, in which case $G\cong{}\nonsplit{C_3}{A_6}$ or $\nonsplit{C_3}{A_7}$, respectively. By inspection, we find that $\nonsplit{C_3}{A_6}$ has a subgroup isomorphic to $A_5$, and that $\nonsplit{C_3}{A_7}$ has no subgroup isomorphic to $A_6$. Hence, $\nonsplit{C_3}{A_6}$ has an action on $3r$ points, and $\nonsplit{C_3}{A_7}$ has no action on $3r$ points. As $\nonsplit{C_3}{A_6}$ is transitive in this action, and all of its proper normal subgroups are abelian (and hence semiregular), this action is semiprimitive. Thus Part $(2.a)$ holds.

Suppose case $(3.)$ of Lemma \ref{trilema} holds. Then Claim $1$ shows that $M$ is isomorphic to $(H')_\delta/(H')_\omega$. Now, since $G^\Delta\cong{}G/M$ and $M\cap{}H'=1$, we have $(H')^\Delta\cong{}H'$. Hence $(H')_\delta\cong{}((H')^\Delta)_\delta\cong{}A_{r-1}$, and so $(H')_\delta$ is either simple or isomorphic to $A_4$, in which case it has a unique nontrivial proper normal subgroup $V\cong{}C_2^2$. As $M$ is isomorphic to a quotient of $(H')_\delta$, this means that either $M\cong{}A_{r-1}$, or $r=5$ and $M\cong{}C_3$, and hence that $G$ is isomorphic to either $A_{r-1}\times{}G^\Delta$, or $C_3\times{}A_r$. Thus Part $(2.b)$ holds.

Suppose case $(4.)$ of Lemma \ref{trilema} holds. Then $M_\sigma\cong{}(H')_\delta$. As $(H')_\delta\cong{}((H^\Delta)')_\delta$ and $((H^\Delta)')_\delta\cong{}A_{r-1}$, we have that $M_\sigma$ is isomorphic to $A_{r-1}$, and so is monolithic. Hence by Lemma \ref{tricor}, $A_{r-1}\lesssim{}T$. In particular, $T$ is nonabelian and simple, so $M\cap{}H=1$ and $[M,H]=1$, and we have that $MH\cong{}M\times{}H\cong{}M\times{}H^\Delta=M\times{}G^\Delta$. Hence Part $(2.c)$ holds, as required.\end{claimproof}\\

\noindent\underline{Claim 4:}
If $H^\Delta=A_\Delta$ and $G^\Delta=S_\Delta$, then Part $(3.)$ of the theorem holds.

\begin{claimproof}
As $G^\Delta/H^\Delta\cong{}C_2$, we must have that $d=1$ if $M$ is abelian, and $d\in\{1,2\}$ if $M$ is nonabelian. As above, we apply Lemma \ref{trilema} to $G^\Delta$. Note that case $(2.)$ cannot occur, as $G^\Delta$ has no elementary abelian normal subgroups.

Suppose case $(1.)$ of Lemma \ref{trilema} holds. Then $M\lesssim{\mathbb{M}(H/M)}=\mathbb{M}(H^\Delta)$, and since $G^\Delta$ has no elementary abelian normal subgroups, case $(1.)$ of Lemma \ref{trilema} holds. Here, $M$ is the unique minimal normal subgroup of $G$. Now $H'=H\cong{}\nonsplit{C_p}{A_\Delta}$ is transitive, and $M\cong{}C_p$ is semiregular. Since $H/M\cong{}H^\Delta\cong{}A_\Delta$ is simple, $M$ is a maximal normal subgroup of $H$. Moreover, $M$ is the unique such subgroup. Any proper normal subgroup of $H$ is contained in $M$, and hence semiregular. Hence every normal subgroup of $H$ is transitive or semiregular, and so $H$ is a semiprimitive group with minimal normal antiplinth $M$, and $H^\Delta=A_\Delta$. As $M\leq{}Z(H)\cap{}H'$, $H$ satisfies all the conditions of Part $(2.a)$ of this theorem, which has already been established in Claim $3$. Hence we may conclude that $H\cong{}\nonsplit{C_3}{A_6}$, and hence that $G\cong{}(\nonsplit{C_3}{A_6})\rtimes{C_2}$.

Suppose case $(3.)$ of Lemma \ref{trilema} holds. Then $M\cong{}(H')_\delta/(H')_\omega$, and in particular that $(H')_\delta\ne{}1$. Hence $(H')_\delta\cong{}A_{r-1}$, is either simple or isomorphic to $A_4$, and by extension $M$ is isomorphic to either $A_{r-1}$, or $C_3$. Hence $G$ is isomorphic to either $(A_{r-1}\times{}A_r)\rtimes{}C_2$, or $(C_3\times{}A_5)\rtimes{}C_2$.

Suppose case $(4.)$ of Lemma \ref{trilema} holds. Then $M_\sigma\cong{}(H')_\delta$. Since $(H')_\delta\cong{}((H^\Delta)')_\delta$ and $((H^\Delta)')_\delta\cong{}A_{r-1}$, we have that $M_\sigma$ is isomorphic to $A_{r-1}$, which is monolithic. Hence $G\cong{}(M\times{}H)\rtimes{}C_2$ is isomorphic to either $(T\times{}A_r)\rtimes{}C_2$, or $(T^2\times{}A_r)\rtimes{}C_2$, and by Lemma \ref{tricor}, $T\gtrsim{}A_{r-1}$.
\end{claimproof}

Now, since $G^\Delta=A_\Delta$ or $S_\Delta$, and $H^\Delta = 1, A_\Delta$, or $S_\Delta$ the theorem follows from Claims $2$, $3$, and $4$. \end{proof} 
The outcomes of Theorem \ref{ansnclassification} can be partitioned into $6$ infinite families and $7$ exceptional groups. We detail these in the tables below:

\begin{center}
\begin{table}
\caption{Infinite Families}
    \begin{tabular}{ | c | c | c | l | c | c |}
    \hline
    $M$                         & $G^\Delta$                & $H^\Delta$        &  Notes                &Innately Transitive                &Case in Theorem \ref{ansnclassification}\\\hline\hline  
    $C_p^d$                     & $A_\Delta$ or $S_\Delta$  & $1$               &  $d\geq{}r-2$         & No                &$1.(a)$\\\hline
    $T^d$                       & $A_\Delta$ or $S_\Delta$  & $1$               &  $d\geq{}r$           & No                &$1.(b)$\\\hline 
    $A_{r-1}$                   & $A_\Delta$ or $S_\Delta$  & $G^\Delta$        &                       & Yes                &$2.(b)$\\\hline
    $A_{r-1}$                   & $S_\Delta$                & $A_\Delta$        &                       & Yes               &$3.(b)$\\\hline
    $T$                         & $A_\Delta$ or $S_\Delta$  & $G^\Delta$        &  $A_{r-1}\lesssim{T}$ & No               &$3.(c)$\\\hline
    $T^2$                       & $S_\Delta$                & $A_\Delta$        &  $A_{r-1}\lesssim{T}$ & No               &$3.(c)$\\\hline
    \end{tabular}
 \end{table}
 \end{center}

\begin{center}
\begin{table}
\label{table2}
\caption{Exceptional Groups}
    \begin{tabular}{ | c | c | c | l | c | c |}
    \hline
    $M$                         & $G^\Delta$                & $H^\Delta$        &  Notes                  & Innately Transitive &Case in Theorem \ref{ansnclassification}\\ \hline\hline
    $C_2^4$                     & $A_7$                     & $1$               &  $d=r-3$                & No          &$1.(a)$\\ \hline
    $C_2^4$                     & $A_8$                     & $1$               &  $d=r-4$, split               & No          &$1.(a)$\\ \hline
    $C_2^4$                     & $A_8$                     & $1$               &  $d=r-4$, nonsplit                & No          &$1.(a)$\\ \hline
    $C_3$                       & $A_6$                     & $A_6$             &  $H$ transitive         & No          &$2.(a)$\\ \hline
    $C_3$                       & $S_6$                     & $A_6$             &  $H$ transitive         & No          &$3.(a)$\\ \hline
    $C_3$                       & $A_5$                     & $A_5$             &  $H$ transitive         & Yes         &$2.(b)$\\ \hline
    $C_3$                       & $S_5$                     & $A_5$             &  $H$ transitive         & Yes         &$3.(b)$\\ \hline
    \end{tabular} 
    \end{table}
\end{center}


\section{Examples}

We will explore the structure of the groups that appear as low-degree exceptions to our results from the previous section. Namely, we construct all groups satisfying the conditions in Table~\ref{table2}, above.

\begin{example}
\label{ex1}
If $G\cong{}C_2^4.A_7$ or $C_2^4.A_8$, then isomorphism types of $G$ are classified by the cohomology groups $H^2(G^\Delta,C_2^4)$, relative to the exceptional irreducible actions of $A_7$ and $A_8$ on $C_2^4$. The isomorphism classes of $H^2(A_7,C_2^4)$ and $H^2(A_8,C_2^4)$ are $1$ and $C_2$, respectively.

If $G\cong{}C_2^4.A_7$, then since $H^2(A_7, C_2^4)\cong{}1$, there exists only one extension, $C_2^4\rtimes{}A_7$. This group contains a subgroup isomorphic to $A_6$, and hence has an action on $2^4\cdot{}7=112$ points, which is readily seen to be semiprimitive.

If $G\cong{}C_2^4.A_8$, then since $H^2(A_8,C_2^4)\cong{}C_2$, there exist two distinct extensions, one of the form $\text{AGL}_4(2)=C_2^4\rtimes{}\text{GL}_4(2)$, and another of the form $\nonsplit{C_2^4}{A_8}$. The second group does not split over $C_2^4$, but does induce $\text{GL}_4(2)$ when it acts on $C_2^4$ by conjugation. On inspection, both the split and nonsplit extensions contain a subgroup isomorphic to $A_7$, and hence each have an action on $2^4\cdot{}8=128$ points, which is semiprimitive.
\end{example}

\begin{center}
\begin{table}
\caption{Semiprimitive actions from low degree linear representations of $A_7$ and $A_8$}
    \begin{tabular}{ | c | c | c | c | c | c | c |}
    \hline
    $G$                                     &  $|G|$                         &  $\text{deg}(G)$       & $b(G)$   & $m(G)$\\\hline\hline  
    $C_2^4\rtimes{}A_7$                     & $2^4\cdot{}7!/2$               &  $2^4\cdot{}7$           &  5        & $100$      \\\hline 
    $C_2^4\rtimes{}A_8\cong\AGL_4(2)$                     & $2^4\cdot{}8!/2$              &  $2^4\cdot{}8$           &   6       & $112$       \\\hline
    $\nonsplit{C_2^4}{A_8}$                 & $2^4\cdot{}8!/2$              &  $2^4\cdot{}8$           &    6      & $112$       \\\hline
    \end{tabular}\\\vspace{2mm}
\end{table}
\end{center}

 \begin{example}
\label{ex2}
If $G=\nonsplit{C_3}{A_6}$, then by the correspondence theorem, $G$ contains a subgroup $H$ for which $Z(G)\leq{}H$ and $H/Z(G)\cong{}A_{5}$. Note that $\mathbb{M}(A_5)=C_2$, and so $H$ must split over $Z(G)\cong{}C_3$. In particular, $G$ contains a subgroup isomorphic to $A_5$, and so has a transitive action on $18$ points, which is semiprimitive.

If $G\cong(\nonsplit{C_3}{A_6})\rtimes{}C_2$, then as $\mathbb{M}(S_6)=C_2$, the minimal normal $C_3$ subgroup must not be central. Hence the involution $C_2$ acts by inversion on $Z(\nonsplit{C_3}{A_6})$. Now, $G$ acts on the $18$ cosets of a subgroup isomorphic to $S_5$, which trivially intersects $Z(G)$. This action is faithful and semiprimitive.
\end{example}

\begin{center}
\begin{table}
\caption{ Semiprimitive actions of groups related to $A_6$ and $S_6$}
    \begin{tabular}{ | c | c | c | c | c | c | c |}
    \hline
    $G$                                         &  $|G|$                        &  $\text{deg}(G)$       & $b(G)$   & $m(G)$\\\hline\hline  
    $\nonsplit{C_3}{A_6}$                       & $3\cdot{}6!/2$                 &  $3\cdot{}6$              & $4$      & $12$      \\\hline 
    $(\nonsplit{C_3}{A_6})\rtimes{}C_2$         & $3\cdot{}6!$                   &  $3\cdot{}6$              & $5$      & $12$     \\\hline
    \end{tabular}\\\vspace{2mm} 
\end{table}
\end{center}

\medskip

\begin{example}\label{ex3}
If $G\cong{}C_3\times{}A_5$, we may consider the action of $G$ on the cosets of $H=\{(\phi(g),g)\mid{}g\in{}A_4\}\leq{}G$, where $\phi:A_4\to{}C_3$ is a homomorphism with kernel $C_2^2$. The action of $G$ on the $15$ cosets of $H$ is semiprimitive. In fact, the subgroup isomorphic to $A_5$ isa  transitive  minimal normal subgroup of $G$, and so $G$ is innately transitive.  Similarly, we may form the group $G\rtimes{}\langle{}\tau\rangle$, where $\tau$ acts on $C_3$ by inversion and on $A_5$ as conjugation by $(1,2)$, and let $G\rtimes{}\langle{}\tau\rangle$ act on the cosets of $H\rtimes{}\langle{}\tau\rangle$. This action is faithful and innately transitive. Viewed another way, these groups are the general and semilinear groups of degree $2$ over $\mathbb{F}_4$, in their natural action on the set of nonzero vectors of $V=\mathbb{F}_4^2$. The base size and minimum degree of these permutation groups can be calculated using {\sc Magma} \cite{MAGMA}, for example.
\end{example}

\begin{center}
\begin{table}
\caption{Semiprimitive groups constructed from $\mathrm{GL}(2,4)$}
    \begin{tabular}{ | c | c | c | c | c | c | c |}
    \hline
    $G$                                     &  $|G|$                         &  $\text{deg}(G)$       & $b(G)$   & $m(G)$\\\hline\hline  
    $C_3\times{}A_5=\GL_2(4)$                        & $3\cdot{}5!/2=(4^2-1)(4^2-4)$                &  $3\cdot{}5=4^2-1$                & $2$        & $12$       \\\hline 
    $(C_3\times{}A_5)\rtimes{}C_2=\GaL_2(4)$          & $3\cdot{}5!=2(4^2-1)(4^2-4)$               &  $3\cdot{}5=4^2-1$                &  $3$        & $12$       \\\hline
    \end{tabular}\\\vspace{2mm}
\end{table}
\end{center}

\section{Bounds}

We will now prove Theorem \ref{bounds}. It will follow directly from Theorems \ref{othm}, \ref{bthm}, \ref{mthm}, \ref{fthm}, and \ref{lthm}, and Corollary \ref{indcor} which establish each establish one part of Theorem \ref{bounds}.

\subsection{Order}
Our first numerical bound is for the order of a semiprimitive group, as a function of its degree.

\begin{theorem}\label{othm}
Let $G\leq{}S_\Omega$ be a semiprimitive permutation group of degree $n$. Then either $|G|<4^n$, or $G\geq{}A_\Omega$.
\end{theorem}

\begin{proof}
If $G$ is quasiprimitive, then the result follows from Theorem \ref{QPD}, part $(1)$. Hence we may assume that $G$ is not quasiprimitive, and so has a nontrivial antiplinth $M$. Suppose $M$ is minimal normal in $G$. Then we may adopt the notation of Section \ref{spn}. Now, $\Delta$ is the set of $M$-orbits, and as $M$ is an antiplinth, we have by Lemma \ref{SMAQP} that $G^\Delta$ is quasiprimitive. If $G^\Delta$ does not contain $A_\Delta$, then by part $(1)$ of Theorem \ref{QPD}, we have $|G^\Delta|<4^{|\Delta|}=4^r$. Hence $|G|=|M||G/M|=|M||G^\Delta|<m4^r$, which is less than or equal to $4^{mr}=4^n$, as required. Hence we may assume that $G^\Delta\geq{}A_\Delta$, and so by Corollary \ref{classcor} that at least one of the following holds:\begin{enumerate}
    \item $r\leq{}4$,
    \item $d\geq{}r-2$,
    \item $T\gtrsim{}A_{r-1}$,
    \item $G$ is isomorphic to one of $C_2^4.A_7$, $C_2^4.A_8$, $\nonsplit{C_3}{A_6}$, $(\nonsplit{C_3}{A_6})\rtimes{}C_2$, $C_3\times{}A_5$,\text{ or } \\ ${(C_3\times{}A_5)\rtimes{}C_2}$.
\end{enumerate}

If $r\leq{}8$, then $|G^\Delta|\leq{}r!<4^r$, as required. Hence we may assume that $r\geq{}9$, and in particular that we are in case $(2.)$ or $(3.)$ of Corollary \ref{classcor}. If we are in case $(2.)$, then $d\geq{}r-2$ and so $m\geq{}2^{r-2}$. If we are in case $(3.)$, then $T\gtrsim{}A_{r-1}$ and so $m\geq{}(r-1)!/2$. In either case, $m>r$, and so by Lemma \ref{indlem}, we have $mr!<4^{mr}$. As $|G|=|M||G/M|\leq{}mr!$ and $mr=n$, we have that $|G|<4^n$, as required. 

Hence we may assume that $M$ in not minimal normal in $G$. This implies that there exists a nontrivial normal subgroup $N\triangleleft{}G$ of $G$ which is properly contained in $M$. Let $N$ be a maximal such subgroup of $M$. Define $\overline{\Omega}:=\Omega/N$ to be the set of orbits of $N$, and $\overline{G}:=G/N$. By Lemma \ref{SPPS}, $G$ acts semiprimitively on $\overline{\Omega}$ with kernel $N$. Hence, $G^{\overline{\Omega}}\cong{}G/N=\overline{G}$. As $N$ is maximal in $M$ with respect to being normal in $G$, $\overline{M}\cong{}M/N$ is a minimal normal subgroup of $\overline{G}\cong{}G/N$. As $M$ is an antiplinth of $G$, $\overline{M}$ is an antiplinth of $\overline{G}$. Hence $\overline{G}$ is a semiprimitive group with a minimal normal antiplinth. From the first part of the proof, we have $|\overline{G}|<4^{|\overline{\Omega}|}=4^{n/|N|}$. Now, $|G|=|N||\overline{G}|<|N|4^{n/|N|}$, which is less than $4^n$, establishing the theorem.
\end{proof}

\subsection{Base Size}

Recall from Definition \ref{defb} that the \textit{base size} $b(G)$ of a permutation group $G$ is the smallest size of a base for $G$. In order to prove Theorem \ref{bthm}, we first need a lemma on the base sizes of certain quotient actions. This lemma may be viewed as a generalisation of \cite{BQG}*{Lemma 5.1}, which establishes the claim when $G$ is a quasiprimitive group (and hence $G_{(\Delta)}$ is trivial).

\begin{lemma}\label{bpb}
Let $G\leq{}S_\Omega$ be a permutation group, and let $\Delta$ be a nontrivial $G$-invariant partition of $\Omega$ with the property that $G_{(\Delta)}$ is semiregular. Then every base $\mathcal{B}$ for $G^\Delta$ on $\Delta$ defines a base $B$ for $G$ on $\Omega$, constructed by taking a set of representatives of the blocks in $\mathcal{B}$. In particular, $$b(G)\leq{}b(G^{\Delta}).$$
\end{lemma}
\begin{proof}
Let $G$, $\Omega$, $\Delta$, $\mathcal{B}$, and $B$ be as above, and let $N=G_{(\Delta)}$ be the kernel of the action of $G$ on $\Delta$. Note that if $\delta\in\Delta$ and $b\in\delta$, then $G_b\leq{}G_{\{\delta\}}$. Thus $$\bigcap_{b\in{B}}G_{b}\leq{}\bigcap_{\delta\in{\mathcal{B}}}G_{\{\delta\}}\leq{}N.$$ However $N$ is semiregular, and hence $N\cap{}G_b$ is trivial for all $b\in{}B$. Thus $$\bigcap_{b\in{B}}G_{b}=1,$$ and so $B$ is a base for $G$. As $B$ has exactly one representative for each $\delta\in\mathcal{B}$, we have $$b(G)\leq{}|B|=|\mathcal{B}|=b(G^\Delta).$$
\end{proof}

\begin{theorem}\label{bthm}
There is a constant $n_1$ such that for every semiprimitive group $G\leq{}S_\Omega$ of degree $n\geq{}n_1$ either $b(G)<4\sqrt{n}\log_2n$, or $G\geq{}A_\Omega$.
\end{theorem}
\begin{proof}
Let $n_0$ be as in Lemma \ref{BQG}, and define $n_1$ to be the least integer such that both $n_0\leq{}4\sqrt{n_1}\log_2(n_1)$ and $n_0\leq{}n_1$ hold. If $G$ is quasiprimitive, then since $n_0\leq{}n_1$, the result follows by Lemma \ref{BQG}. Hence we may suppose that $G$ is not quasiprimitive, and so has a nontrivial antiplinth $M$. 

Suppose $M$ is a minimal normal subgroup of $G$. We adopt the notation of Section \ref{spn}. Now $\Delta$ is the set of $M$-orbits. By Lemma \ref{bpb}, $b(G)\leq{}b(G^\Delta)$. We proceed by bounding $b(G^\Delta)$. As $M$ is an antiplinth, we have by Lemma \ref{SMAQP} that $G^\Delta$ is quasiprimitive. If $G^\Delta$ does not contain $A_\Delta$, then by Lemma \ref{BQG}, $b(G^\Delta)<4\sqrt{r}\log_2r$, and hence $b(G^\Delta)<4\sqrt{r}\log_2r<4\sqrt{n}\log_2n$, as required. Thus we may assume $G^\Delta\geq{}A_\Delta$. By Corollary \ref{classcor}, one of the following holds:\begin{enumerate}
    \item $r\leq{}4$,
    \item $d\geq{}r-2$,
    \item $T\gtrsim{}A_{r-1}$,
    \item $G$ is isomorphic to one of $C_2^4.A_7$, $C_2^4.A_8$, $\nonsplit{C_3}{A_6}$, $(\nonsplit{C_3}{A_6})\rtimes{}C_2$, $C_3\times{}A_5$,\text{ or } \\ {$(C_3\times{}A_5)\rtimes{}C_2$}. \end{enumerate} 

Note that, regardless of the isomorphism type of $G^\Delta$, we have that $b(G^\Delta)\leq{}r-1$. Suppose first that $(1.)$ holds. Then $b(G^\Delta)\leq{}3\leq{}4\sqrt{2}\log_2{2}<4\sqrt{n}\log_2n,$ establishing the theorem in this case. Now, suppose that $(2.)$ holds. Here $m\geq{}2^{r-2}$, and so $n=mr\geq{}r2^{r-2}$. Hence, $$b(G^\Delta)\leq{}r-1<4\sqrt{r2^{r-2}}\log_2{(r2^{r-2})}\leq{}4\sqrt{n}\log_2n.$$ If $(3.)$ holds, then $m\geq{}|A_{r-1}|=(r-1)!/2$, and so $n=mr\geq{}r!/2$. Hence, $$b(G^\Delta)\leq{}r-1<4\sqrt{\frac{r!}{2}}\log_2{\left(\frac{r!}{2}\right)}\leq{}4\sqrt{n}\log_2n.$$ Finally, suppose $(4.)$ holds, and $G$ is isomorphic to one of 
$C_3\times{}A_5$
 $(C_3\times{}A_5)\rtimes{}C_2$,
  $\nonsplit{C_3}{A_6}$,
 $(\nonsplit{C_3}{A_6})\rtimes{C_2}$,
  $C_2^4.A_7$,,
    or 
     $C_2^4.A_8$. Now, $G$ is detailed in Example \ref{ex1}, \ref{ex2}, or \ref{ex3}, and in particular, $r\in{}\{5,6,7,8\}$, and $n\in\{15, 18, 112, 128\}$. Hence $$b(G^\Delta)\leq{}r-1\leq{}7<4\sqrt{15}\log_215\leq{}4\sqrt{n}\log_2n.$$

Hence we may assume that $M$ is not minimal normal in $G$. This implies that there exists a nontrivial normal subgroup $N\triangleleft{}G$ of $G$ which is properly contained in $M$. Let $N$ be a maximal such subgroup of $M$. Define $\overline{\Omega}:=\Omega/N$ to be the set of orbits of $N$, and $\overline{G}:=G/N$. By Lemma \ref{SPPS}, $G$ acts semiprimitively on $\overline{\Omega}$ with kernel $N$. Hence, $G^{\overline{\Omega}}\cong{}G/N=\overline{G}$. As $N$ is maximal in $M$ with respect to being normal in $G$, $\overline{M}\cong{}M/N$ is a minimal normal subgroup of $\overline{G}\cong{}G/N$. As $M$ is an antiplinth of $G$, $\overline{M}$ is an antiplinth of $\overline{G}$. Hence $\overline{G}$ is a semiprimitive group with a minimal normal antiplinth. By the first part of the proof, we have that $b(\overline{G})<4\sqrt{\overline{\Omega}}\log_2|\overline{\Omega}|$, and hence that $$b(G)\leq{}b(\overline{G})<4\sqrt{|\overline{\Omega}|}\log_2|\overline{\Omega}|<4\sqrt{n}\log_2n.$$ The theorem follows.
\end{proof}

\begin{corollary}
There is a constant $n_1$ such that if $G\leq{}S_\Omega$ is a semiprimitive permutation group of degree $n\geq{}n_1$, then either $|G|< 2^{4\sqrt{n}(\log_2n)^2}$ or $G\geq{}A_\Omega$.
\end{corollary}

\begin{proof}
If $\mathfrak{B}$ is a base for $G$ of size $b(G)$, then there exists an injection $G\to{}\Omega^{b(G)}$ given by choosing an ordering $\{b_i\}$ for $\mathfrak{B}$, and setting $g\mapsto{}(b_1g, ..., b_{m(G)}g)$. Hence $$|G|\leq{}|\Omega|^{b(G)}\leq{}n^{4\sqrt{n}(\log_2n)}=2^{4\sqrt{n}(\log_2n)^2}.$$
\end{proof}

\subsection{Minimal Degree}
Recall from Definition \ref{defm} the \textit{minimal degree} $m(G)$ of a permutation group $G$ is the least number of points moved by a nontrivial element of $G$.
In order to prove Theorem \ref{mthm}, we first need a lemma relating the minimal degree of a group to the minimal degree of its quotient actions. This lemma can be considered a generalisation of \cite{BQG}*{Lemma $7.1$}, which establishes the claim for quasiprimitive groups.

\begin{lemma}\label{medlem}\label{medlem2}
Let $G\leq{}S_\Omega$ be a nontrivial semiprimitive permutation group, and $\Delta$ be block system for $G$, with blocks of size $s$. Then $m(G)\geq{}s.m(G^\Delta).$
\end{lemma}
\begin{proof}
Let $M:=G_{(\Delta)}$ be the kernel of the action of $G$ on $\Delta$. Note that since $G$ is semiprimitive, $M$ is semiregular. Let $1\ne{}g\in{}G$. If $g\in{}M$, then since $M$ is semiregular, $g$ moves every point in $\Omega$. If $g\in{}G\backslash{}M$, then $g^\Delta$ is a nontrivial element of $G^\Delta$, and hence moves at least $m(G^\Delta)$ of the blocks in $\Delta$. As each block has $s$ elements, the lemma follows.
\end{proof}

Note that if we adopt the convention that $m(G)=0$ if $G$ is trivial, then we can remove the word `nontrivial' from the statement of Lemma \ref{medlem}.

\begin{theorem}\label{mthm}
Let $G$ be a semiprimitive permutation group of degree $n$. Then either $m(G)\geq(\sqrt{n}-1)/2$, or $G\geq{}A_\Omega$.
\end{theorem}

\begin{proof}
If $G$ is quasiprimitive, then the result follows from Theorem \ref{BQG}, part $(3)$. Hence we may assume that $G$ is not quasiprimitive, and so $G$ has a nontrivial antiplinth $M$.

Suppose $M$ is minimal normal in $G$. We adopt the notation of Section \ref{spn}. As $M$ is an antiplinth, we have by Lemma \ref{SMAQP} that $G^\Delta$ is quasiprimitive. If $G^\Delta$ does not contain $A_\Delta$, then by Lemma \ref{BQG}, $m(G^\Delta)\geq{}(\sqrt{r}-1)/2$. By Lemma \ref{medlem2}, $m(G)\geq{}|M|m(G^\Delta)$. Hence, $m(G)\geq{}m(\sqrt{r}-1)/2,$ which by Lemma \ref{numlem2} is greater than $(\sqrt{mr}-1)/2=(\sqrt{n}-1)/2$, as required. Hence we may assume that $G^\Delta\geq{}A_\Delta$. By Corollary \ref{classcor}, one of the following holds: \begin{enumerate}
    \item $r\leq{}4$,
    \item $d\geq{}r-2$,
    \item $T\gtrsim{}A_{r-1}$ 
    \item $G$ is isomorphic to one of $C_2^4.A_7,  C_2^4.A_8, \nonsplit{C_3}{A_6}, (\nonsplit{C_3}{A_6})\rtimes{}C_2, C_3\times{}A_5,\text{ or } \\ {(C_3\times{}A_5)\rtimes{}C_2}.$ 
\end{enumerate}

Suppose first that $r\leq{}4$. By Lemma \ref{medlem2}, $m(G)\geq{}|M|m(G^\Delta)\geq{}2m$, and since $m\geq{}2$ we have $2m\geq{}(\sqrt{4m}-1)/2\geq{}(\sqrt{rm}-1)/2=(\sqrt{n}-1)/2$.

Now suppose either $d\geq{}r-2$ or $T\gtrsim{A_{r-1}}$, or that $(M,r)$ is $(C_2^4, 7)$ or $(C_2^4, 8)$. Now, $m>r$, and so $m>\sqrt{rm}$, which implies $2m>2(\sqrt{rm}-1)\geq{}(\sqrt{n}-1)/2.$ If $(M,r)$ is $(C_3,5)$, then $2m=6>(\sqrt{15}-1)/2=(\sqrt{n}-1)/2$.

Finally, suppose that $(M,r)$ is $(C_3,6)$. By surveying all such groups in Magma \cite{MAGMA}, we find that either $G^\Delta=A_\Delta$ and $m(G)=9$, or $G^\Delta=S_\Delta$ and $m(G)=6$. In either case, $m(G)\geq{}6>(\sqrt{18}-1)/2\sim{}1.62$, and so the bound is satisfied.

Hence we may assume that $M$ is not minimal normal in $G$. Let $N$ be a normal subgroup of $G$ contained in $M$ such that $M/N$ is a minimal normal subgroup of $G/N$. Let $\overline{\Omega}$ be the set of $N$-orbits in $\Omega$. Now, by Lemma \ref{SPPS}, $G$ acts semiprimitively on $\overline{\Omega}$ with kernel $N$. Hence $G^{\overline{\Omega}}$ is a semiprimitive group with a minimal normal subgroup $M^{\overline{\Omega}}$. As $M$ is an antiplinth of $G$, $M^{\overline\Omega}$ is an antiplinth of $G^{\overline{\Omega}}$. Hence by the first part of the proof, $m(G^{\overline{\Omega}})\geq{}(\sqrt{n/|N|}-1)/2$. By Lemma \ref{medlem}, $$m(G)\geq{}|N|m(G^{\overline\Omega})\geq{}|N|(\sqrt{n/|N|}-1)/2,$$ which by Lemma \ref{numlem2} is greater than $(\sqrt{n}-1)/2$, as required.
\end{proof}

Classes of permutation groups $\mathbb{H}$ for which the minimal degree $m(H)$ of $H\in\mathbb{H}$ is not bounded in $\text{deg}(H)$ by a constant have been shown by \cite{KPS} to be \textit{indistinguishable}. The significance of distinguishability is that there exists an efficient quantum algorithm for solving the hidden subgroup problem in $S_n$ when the hidden subgroup is drawn from a distinguishable class of groups. The interested reader is directed toward \cite{KPS} for more information.

\begin{corollary}\label{indcor}
Semiprimitive groups are indistinguishable subgroups of $S_\Omega$. 
\end{corollary}
\begin{proof}
This follows directly from Theorem \ref{mthm} and \cite{KPS}*{Theorem C}.
\end{proof}

\subsection{Fixed Point Ratio}

Recall from Definition \ref{deffix} that the fixity $\text{fix}(G)$ of $G$ is the maximum number of points fixed by a nontrivial element of $G$, and the fixed point ratio $\text{fpr}(G)$ of $G$ is the ratio of the fixity and degree of $G$.

\begin{theorem}\label{fthm}
Let $G\leq{}S_\Omega$ be a semiprimitive permutation group of degree $n$ with an antiplinth $M$ for which the socle of $G/M$ is not a product of alternating groups. Then $\text{fpr}(G)\leq{}4/7$.
\end{theorem}
\begin{proof}
If $G$ is primitive, the result follows from Theorem \ref{gmf} with $M=1$. Hence we may assume that $G$ is not primitive.

Suppose $G$ is quasiprimitive. Then, all nontrivial normal subgroups of $G$ are transitive, and hence $M=1$. This implies $G\cong{}G/M$, and so the socle of $G$ is not a product of alternating groups. As $G$ is  not primitive, it has a nontrivial system of imprimitivity, and we may choose a maximal such system $\Delta$. Let $r=|\Delta|$. As $G$ is quasiprimitive, it acts faithfully on $\Delta$. Hence $G\cong{}G^\Delta$, and in particular $G$ and $G^\Delta$ have isomorphic socles. In particular, the socle of $G^\Delta$ is not a product of alternating groups. As $\Delta$ is a maximal system of imprimitivity, $G^\Delta$ is primitive, and so by Theorem \ref{gmf} we have that $\text{fpr}(G^\Delta)\leq{}4/7$. Let $1\ne{}g\in{}G$. Then since $\text{fpr}(G^\Delta)\leq{}4/7$, $g^\Delta$ fixes at most $4r/7$ points in $\Delta$. As each block in $\Delta$ has $n/r$ points in it, we get that $g$ fixes at most $(4r/7)(n/r)=4n/7$ points in $\Omega$, as required. Hence we may suppose that $G$ is not quasiprimitive.

As $G$ is not quasiprimitive, it has nontrivial intransitive normal subgroups, and in particular $M$ is nontrivial. Let $|M|=m$. Since $G$ is semiprimitive and $M$ is intransitive, $M$ is semiregular. Let $\Sigma$ be the set of $M$-orbits in $\Omega$, and let $s=|\Sigma|$. Since $M$ is semiregular, $s=n/m$. As $M$ is a maximal intransitive normal subgroup of the semiprimitive group $G$, we have by Lemma \ref{SPPS} that $G^\Sigma\cong{}G/M$, and by Lemma \ref{SMAQP} that $G^\Sigma$ is quasiprimitive. From the first part of the proof, we conclude that $\text{fpr}(G^\Sigma)\leq{}4/7$. Let $1\ne{}g\in{}G$. If $g$ is in $M$, then since $M$ is semiregular, $g$ moves every point in $\Omega$, and so has no fixed points. If $g$ is not in $M$, then since $G^\Sigma\cong{}G/M$, we have that $g^\Sigma$ is a nontrivial element of $G^\Sigma$. Since $\text{fpr}(G^\Sigma)\leq{}4/7$, we conclude that $g^\Sigma$ fixes at most $4s/7$ elements of $\Sigma$. Each element of $\Sigma$ is a set of $m=n/s$ elements from $\Omega$, and hence $g$ fixes at most $(4s/7)(n/s)=4n/7$ points, as required.\end{proof}

\subsection{Chief Length}

Recall from Definition \ref{defc} that the \textit{chief length} $l(G)$ of a group $G$ is the length of a chief series for $G$, and that all chief series for a given finite group have the same length.

\begin{theorem}\label{lthm}
Let $G\leq{}S_\Omega$ be a semiprimitive group of degree $n$. Then the chief length $l(G)$ of $G$ satisfies $l(G)\leq{}2\log_2n.$
\end{theorem}
\begin{proof}
Let $M$ be an antiplinth of $G$, and let $\Delta$ be the set of $M$-orbits. Let $m:=|M|$ and $r:=|\Delta|$. Without loss of generality, we may consider a chief series for $G$ passing through $M$. As each factor has order at least $2$, there are at most $\log_2m$ terms in the chief series up to and including $M$. By Lemma \ref{SMAQP}, $G^\Delta\cong{}G/M$ is quasiprimitive, and so by Theorem \ref{BQG} we have that $l(G^\Delta)\leq{}2\log_2r$. Hence there are at most $2\log_2r$ terms in the chief series from $M$ to $G$. Now, $$l(G)\leq{}l(G^\Delta)+\log_2m\leq{}2\log_2r+\log_2m<2\log_2(rm)=2\log_2n,$$ as required.\end{proof}

\section*{Acknowledgments}
The first  author gratefully acknowledges the support of the Australian Research Council Discovery Grant DE160100081.  The second author gratefully acknowledges the support of the Australian Research Council Discovery Grant DP160102323. The third author thanks the Australian Government for the support of a Research Training Program grant.



\begin{bibdiv}
\begin{biblist}

\bib{PrimB}{article}{
      author={Babai, L\'aszl\'o},
       title={On the order of uniprimitive permutation groups},
        date={1981},
     journal={Ann. of Math. (2)},
      volume={113},
      number={3},
       pages={553\ndash 568},
}

\bib{JB}{article}{
      author={Bamberg, John},
       title={Bounds and quotient actions of innately transitive groups},
        date={2005},
     journal={J. Aust. Math. Soc.},
      volume={79},
      number={1},
       pages={95\ndash 112},
}

\bib{BM}{article}{
      author={Bereczky, \'Aron},
      author={Mar\'oti, Attila},
       title={On groups with every normal subgroup transitive or semiregular},
        date={2008},
     journal={J. Algebra},
      volume={319},
      number={4},
       pages={1733\ndash 1751},
}

\bib{MAGMA}{article}{
      author={Bosma, Wieb},
      author={Cannon, John},
      author={Playoust, Catherine},
       title={The {M}agma algebra system. {I}. {T}he user language},
        date={1997},
        ISSN={0747-7171},
     journal={Journal of Symbolic Computation},
      volume={24},
      number={3-4},
       pages={235\ndash 265},
         url={http://dx.doi.org/10.1006/jsco.1996.0125},
        note={Computational algebra and number theory (London, 1993)},
}

\bib{CSPGGR}{article}{
      author={Giudici, Michael},
      author={Morgan, Luke},
       title={A class of semiprimitive groups that are graph-restrictive},
        date={2014},
     journal={Bull. Lond. Math. Soc.},
      volume={46},
      number={6},
       pages={1226\ndash 1236},
}

\bib{LSGTW}{article}{
      author={Giudici, Michael},
      author={Morgan, Luke},
       title={On locally semiprimitive graphs and a theorem of {W}eiss},
        date={2015},
     journal={J. Algebra},
      volume={427},
       pages={104\ndash 117},
}

\bib{GM}{article}{
      author={Giudici, Michael},
      author={Morgan, Luke},
       title={{A theory of semiprimitive groups}},
        date={2018},
     journal={J. Algebra},
      number={503},
       pages={146\ndash 185},
}

\bib{Me}{article}{
      author={Glasby, S.~P.},
      author={Praeger, Cheryl~E.},
      author={Rosa, Kyle},
      author={Verret, Gabriel},
       title={Bounding the composition length of primitive permutation groups
  and completely reducible linear groups},
     journal={To appear in the Journal of the London Mathematical Society,
  https://doi.org/10.1112/jlms.12138},
}

\bib{gmf}{article}{
      author={Guralnick, Robert},
      author={Magaard, Kay},
       title={On the minimal degree of a primitive permutation group},
        date={1998},
        ISSN={0021-8693},
     journal={Journal of Algebra},
      volume={207},
      number={1},
       pages={127\ndash 145},
         url={http://dx.doi.org/10.1006/jabr.1998.7451},
}

\bib{KPS}{article}{
      author={Kempe, Julia},
      author={Pyber, L\'aszl\'o},
      author={Shalev, Aner},
       title={Permutation groups, minimal degrees and quantum computing},
        date={2007},
     journal={Groups Geom. Dyn.},
      volume={1},
      number={4},
       pages={553\ndash 584},
}

\bib{bettermed}{article}{
      author={Liebeck, Martin~W.},
      author={Saxl, Jan},
       title={Minimal degrees of primitive permutation groups, with an
  application to monodromy groups of covers of {R}iemann surfaces},
        date={1991},
        ISSN={0024-6115},
     journal={Proceedings of the London Mathematical Society. Third Series},
      volume={63},
      number={2},
       pages={266\ndash 314},
         url={http://dx.doi.org/10.1112/plms/s3-63.2.266},
}

\bib{NM}{article}{
      author={Menezes, Nina~E.},
       title={Random generation and chief length of finite groups},
        date={2013},
     journal={PhD Thesis, University of St Andrews},
}

\bib{OnanScott}{article}{
      author={Praeger, Cheryl~E.},
       title={An {O}'{N}an-{S}cott theorem for finite quasiprimitive
  permutation groups and an application to {$2$}-arc transitive graphs},
        date={1993},
     journal={J. London Math. Soc. (2)},
      volume={47},
      number={2},
       pages={227\ndash 239},
}

\bib{PrimO}{article}{
      author={Praeger, Cheryl~E.},
      author={Saxl, Jan},
       title={On the orders of primitive permutation groups},
        date={1980},
     journal={Bull. London Math. Soc.},
      volume={12},
      number={4},
       pages={303\ndash 307},
}

\bib{PraegerSchneider}{book}{
      author={Praeger, Cheryl~E.},
      author={Schneider, Csaba},
       title={Permutation groups and cartesian decompositions},
      series={London Mathematical Society Lecture Note Series},
   publisher={Cambridge University Press},
        date={2018},
      volume={449},
}

\bib{BQG}{article}{
      author={Praeger, Cheryl~E.},
      author={Shalev, Aner},
       title={Bounds on finite quasiprimitive permutation groups},
        date={2001},
        ISSN={1446-7887},
     journal={Journal of the Australian Mathematical Society},
      volume={71},
      number={2},
       pages={243\ndash 258},
         url={http://dx.doi.org/10.1017/S1446788700002895},
        note={Special issue on group theory},
}

\bib{Rotman}{book}{
      author={Rotman, Joseph~J.},
       title={A first course in abstract algebra},
   publisher={Prentice Hall, Inc., Upper Saddle River, NJ},
        date={1996},
}

\bib{OGRPG}{article}{
      author={Spiga, Pablo},
      author={Verret, Gabriel},
       title={On intransitive graph-restrictive permutation groups},
        date={2014},
     journal={J. Algebraic Combin.},
      volume={40},
      number={1},
       pages={179\ndash 185},
}

\bib{mrd2}{article}{
      author={Wagner, Ascher},
       title={The faithful linear representation of least degree of {$S_{n}$}
  and {$A_{n}$} over a field of characteristic {$2.$}},
        date={1976},
        ISSN={0025-5874},
     journal={Mathematische Zeitschrift},
      volume={151},
      number={2},
       pages={127\ndash 137},
}

\bib{mrd1}{article}{
      author={Wagner, Ascher},
       title={The faithful linear representations of least degree of {$S_{n}$}
  and {$A_{n}$} over a field of odd characteristic},
        date={1977},
        ISSN={0025-5874},
     journal={Mathematische Zeitschrift},
      volume={154},
      number={2},
       pages={103\ndash 114},
}

\bib{Wie}{book}{
      author={Wielandt, Helmut},
       title={Finite permutation groups},
   publisher={Academic Press, New York-London},
        date={1964},
}

\bib{PMN}{book}{
      author={Wielandt, Helmut},
       title={Mathematische {W}erke/{M}athematical works. {V}ol. 1},
   publisher={Walter de Gruyter \& Co., Berlin},
        date={1994},
        note={Group theory, With essays on some of Wielandt's works by G.
  Betsch, B. Hartley, I. M. Isaacs, O. H. Kegel and P. M. Neumann, Edited and
  with a preface by Bertram Huppert and Hans Schneider},
}

\bib{TFSG}{book}{
      author={Wilson, Robert~A.},
       title={The finite simple groups},
      series={Graduate Texts in Mathematics},
   publisher={Springer-Verlag London, Ltd., London},
        date={2009},
      volume={251},
        ISBN={978-1-84800-987-5},
         url={http://dx.doi.org/10.1007/978-1-84800-988-2},
}

\end{biblist}
\end{bibdiv}
\end{document}